\documentclass{llncs}
\usepackage{amsmath,amssymb}
\usepackage{subfigure}
\usepackage{tikz}
\usepackage[linesnumbered,ruled,vlined]{algorithm2e}
\newcommand{\classNP}{\mathcal{NP}}
\newcommand{\classP}{\mathcal{P}}
\usetikzlibrary[backgrounds]
\usepackage{array}
\usepackage{nicefrac}
\usepackage{hyperref}

\newcommand{\abs}[1]{\left\lvert#1\right\rvert}
\newcommand{\sbmultcov}{\ensuremath{set \, multicover}\xspace}

\newcommand{\bb}{\mathbf{b}}
\newcommand{\kk}{\mathbf{k}}

\begin{document}
\title{ Approximation algorithm
for the Multicovering Problem }

\author{Abbass Gorgi
\and Mourad El Ouali
\and Anand Srivastav
\and Mohamed Hachimi}

\institute{%
Abbass Gorgi \\
University Ibn Zohr,
Agadir,Morocco \\
\email{abbass.gorgi@gmail.com}
\and
Mourad El Ouali\\ 
Christian Albrechts university Kiel,
 Germany\\
\email{Elouali@math.uni-kiel.de}
\and
Anand Srivastav\\ 
Christian Albrechts university Kiel,
 Germany\\
\email{srivastavi@math.uni-kiel.de}
\and
Mohamed Hachimi\\ 
University Ibn Zohr,
Agadir,Morocco \\
\email{m.hachimi@uiz.ac.ma}
}

\date{Received: date / Accepted: date}

\maketitle

\begin{abstract}
Let $\mathcal{H}=(V,\mathcal{E})$ be a hypergraph
with maximum edge size $\ell$ and maximum degree $\Delta$. 
For given numbers $b_v\in \mathbb{N}_{\geq 2}$, $v\in V$,
a set multicover in $\mathcal{H}$ is a
set of edges $C \subseteq \mathcal{E}$ such that every vertex $v$ in $V$
belongs to at least $b_v$ edges in $C$.
\sbmultcov is the problem of finding a minimum-cardinality set multicover.
Peleg, Schechtman and Wool conjectured that for any fixed $\Delta$
and $b:=\min_{v\in V}b_{v}$, the problem of \sbmultcov is not
approximable within a ratio less
than $\delta:=\Delta-b+1$, unless $\classP =\classNP$. Hence it's a challenge to explore for which classes of hypergraph the conjecture doesn't hold. 
 
We present a polynomial time algorithm for the \sbmultcov problem
which combines a deterministic threshold algorithm with
conditioned randomized rounding steps.
Our algorithm yields an approximation ratio of
$ \max\left\{ \frac{148}{149}\delta, \left(1- \frac{ (b-1)e^{\frac{\delta}{4}}}{94\ell} \right)\delta \right\}$. Our  result not only improves over the approximation ratio presented by Srivastav et al (Algorithmica 2016) but it's more general since we set no restriction on the parameter $\ell$.\\
Moreover we present a further polynomial time algorithm with an approximation ratio of $\frac{5}{6}\delta$ for hypergraphs with $\ell\leq (1+\epsilon)\bar{\ell}$ for any fixed 
$\epsilon \in [0,\frac{1}{2}]$, where $\bar{\ell}$ is the average edge size.
The analysis of this algorithm relies on matching/covering duality due to Ray-Chaudhuri (1960), 
which we convert into an approximative form.
The second performance disprove the conjecture of peleg et al for a large subclass of hypergraphs.
 
\end{abstract}


\section{Introduction}
The \sbmultcov problem is a fundamental covering issue that widely explored in the theory of optimization. A nicely formulation of this problem may given by the notion of hypergraphs which offer tools to deal with sets.

A hypergraph is a pair 
$\mathcal{H}=(V,\mathcal{E})$, where $V$ is a finite set and 
$\mathcal{E}\subseteq 2^V$ is a family of some subsets of $V$. We call the elements of $V$
vertices and the elements of $\mathcal{E}$ (hyper-)edges. Further let 
$n := |V|$, $m := |{\cal E}|$.
W.l.o.g. let the elements of $V$ be enumerated as $v_1,v_2,\dots,v_n$.
Let $\ell$ be the maximum edge size, $\bar{\ell}=\frac{1}{m}\sum^{m}_{j=1}|E_{j}|$ the average edge size and let $\Delta$ be the maximum vertex degree,
where the degree of a vertex is the number of edges containing that vertex. If for every $E\in \mathcal{E};\, |E|=\ell$ than the hypergraph is called uniform.

\noindent Let $\bb:=(b_1,b_2,\dots,b_n) \in \mathbb{N}_{\geq 2}^{n}$ be given.
If a vertex $v_i$, $i\in[n]$, is contained in at least $b_i$ edges
of some subset $C \subseteq \mathcal{E}$,
we say that the vertex $v_i$ is fully covered by $b_i$ edges in $C$.
A set multicover in $\mathcal{H}$ is a set of edges $C \subseteq \mathcal{E}$ 
such that every vertex $v_i$ in $V$ is fully covered by $b_i$ edges in $C$.
\sbmultcov problem is the task of finding a set multicover of minimum cardinality.
Note that the usual set cover problem, which is known to be NP-hard \cite{Karp}, is a special case with $b_i=1$ for all $i\in [n]$. Furthermore Peleg, Schechtman and Wool conjectured that for any fixed $\Delta$ and $b:=\min_{i\in [n]}b_i$ the problem cannot be approximated by a ration smaller than $\delta:= \Delta-b+1$ unless $\mathcal{P}=\mathcal{NP}$. Hence it remained an open problem whether an approximation ration of $\beta\delta$ with $\beta<1$ constant can be proved. We say that an algorithm $A$ is
an approximation algorithm for \sbmultcov problem 
with performance ratio $\alpha$,
if for each instance $I$ of size $n$ of \sbmultcov problem,
$A$ runs in polynomial time in $n$
and returns a value $\abs{A(I)}$
such that $\abs{A(I)} \le \alpha\cdot\mathrm{Opt}$, where $\mathrm{Opt}$ is the cardinality of a minimum set multicover.

The \sbmultcov problem 
can also be formulated as an integer linear program as follows
\begin{equation*}
\min \{\sum_{j=1}^m x_j;\, Ax\geq \bb, x\in \{0,1\}^{m}\}\qquad \left(\mbox{ILP}(\Delta, {\bf b})\right)
\end{equation*}
where $ A=(a_{ij})_{i \in [n], \,j \in [m]}\in \{0,1\}^{n \times m}$
is the vertex-edge incidence matrix of ${\cal H}$
and ${\bf b}=(b_1,b_2,\dots,b_n) \in \mathbb{N}_{\geq 2}^{n}$
is the given integer vector.

The linear programming relaxation LP($\Delta,\,{\bf b}$)
of ILP($\Delta,\,{\bf b}$) is given by relaxing the integrality constraints
to  $x_j\in [0,1] \;$for all $j \in [m]$.
Let $x^{*}\in [0,1]^{m}$ be an optimal solution of LP($\Delta,\,{\bf b}$) than ${\rm Opt}^{*} =\sum_{j=1}^mx^{*}_{j}$ is the value of the optimal solution to LP($\Delta,\,{\bf b}$), We have ${\rm Opt}^{*} \leq \mathrm{Opt}$.

\noindent {\bf Related Work.} The set cover problem $(b=1)$ has been over decades intensively explored. Several deterministic approximation algorithms are exhibited for
this problem \cite{Ba01,GaKhSr01,Hoch82,Koufogiannakis},
all with approximation ratios $\Delta$. On the other hand Khot and Regev in \cite{KR08} proved that the problem
cannot be approximated within factor $\delta-\epsilon =\Delta-\epsilon$ assuming that unique games conjecture is true. Furthermore Johnson \cite{Johnson74}
and Lov\'asz \cite{Lovasz94} gave a greedy algorithm
with performance ratio $H(\ell)$,
where $ H(\ell) =\sum_{i=1}^{\ell}\frac{1}{i}$ is the harmonic number.
Notice that $H(\ell)\leq 1+\ln(\ell)$.
For hypergraphs with bounded $\ell$, Duh and Fürer \cite{Duh97}
used the technique called semi-local optimization
improving $H(\ell)$ to $H(\ell)-\frac{1}{2}$. 
In contrast to set cover problem it is less known for the case $b\geq 2$. Let give a brief summary of the known approximability results.
In paper \cite{Vazirani01}, Vazirani using dual fitting method extended the result of Lov\'asz \cite{Lovasz94} for $b\geq 1$. Later Fujito et al. \cite{Fujito} improved the algorithm of Vazirani and achieved an approximation ratio of $H(\ell) -\frac{1}{6}$
for $\ell$ bounded.
Hall and Hochbaum \cite{HH86} achieved by a greedy algorithm 
based on LP duality an approximation ratio of $\Delta$.
By a deterministic threshold algorithm
Peleg, Schechtman and Wool in 1997 \cite{PSW97,PSW93}
improved this result and gave an approximation ratio $\delta$.
They were also the first to propose an approximation algorithm
for the \sbmultcov problem with approximation ratio below $\delta$,
namely a randomized rounding algorithm
with performance ratio $(1-(\frac{c}{n})^\frac{1}{\delta})\cdot\delta$
for a small constant $c>0$.
However, their ratio is depending on $n$
and asymptotically tends to $\delta$.
A randomized algorithm of hybrid type was later given
by Srivastav et al \cite{EMS16}.
Their algorithm achieves for hypergraphs with
$l \in \mathcal{O}\left(\max\{(nb)^\frac{1}{5},n^\frac{1}{4}\}\right)$ an approximation ratio of
$ \left(1-\frac{11 (\Delta - b)}{72l}\right)\cdot \delta $
with constant probability.

\noindent {\bf Our Results.} 
The main contribution of our paper is the combination
of a deterministic threshold-based algorithm
with conditioned randomized rounding steps.
The idea is to algorithmically discard instances
that can be handled deterministically
in favor of instances
for which we obtain a constant factor approximation less than $\delta$
using a randomized strategy.

In the following we give a brief overview of the method.
First we give some fundamentals results based on the LP relaxation with threshold
that allows us to come up with an approximation ratio
strictly less than $\delta$
and use this results for the first algorithm.
This is an extension of an algorithm by Hochbaum \cite{Hoch82}
for the set cover problem and the vertex cover problem.

Let $(x^{\ast}_{1},\cdots,x^{\ast}_{m})$ be the optimal solution of the LP($\Delta,\,{\bf b}$). We define $C_1:=\{E_j \in \mathcal{E} \mathrel{|} x_j^{\ast}\geq\frac{2}{\delta +1}\}$, $C_2:=\{E_j\in \mathcal{E}\mathrel{|}\frac{1}{\delta} \leq x_j^{\ast}<\frac{2}{\delta +1}\}$ 
and $C_3:=\{E_j\in \mathcal{E}\mathrel{|} 0<x_j^{\ast}<\frac{2}{\delta +1}\}$. It follows that $C_1\cap C_2 =\emptyset $
and  $C_1\cup C_2 $ is a feasible set multicover.

Our first algorithm is designed as a cascade
of a deterministic and a randomized rounding step
followed by greedy repairing.
The threshold type algorithm first solves the relaxed LP($\Delta,\,{\bf b}$)
problem and then picks all the edges corresponding to variables
with fractional values at least $\frac{2}{\delta +1}$ to the output set. 
Depending on the cardinality of sets $C_1$ and $C_2$,
we use LP-rounding with randomization for the hyperedges of the set $C_3$.
Every edge of $C_3$ is independently added to the output set
with probability $\frac{\delta +1}{2} x_j^*$.
To guarantee feasibility, 
we proceed with a repairing step.
Our algorithm is an extension of an example
given in \cite{EFS14,EFS14a,EMS16,GaKhSr01,HH86,PSW93}
for the vertex cover, partial vertex cover and \sbmultcov problem
in graphs and hypergraphs.
 
The methods used in this paper rely on an application
of the Chernoff–Hoeffding bound technique
for sums of independent random variables
and are based on estimating the variance of the summed random variables
for invoking the Chebychev-Cantelli inequality.

We give a detailed analysis of the first algorithm
in which we explore the cases
by comparing the cardinality of the two sets $C_1$ and $C_2$
and the relative cardinality of $C_1$ with respect to ${\rm Opt}^{*}$.
Our algorithm yields a performance ratio of
$ \max\left\{ \frac{148}{149}\delta, \left(1- \frac{(b-1) e^{\frac{\delta}{4}}}{94\ell} \right)\delta \right\}$. 
This ratio means a constant factor less than $\delta$
for many settings of the parameters $\delta$, $b$ and $\ell$.
Further it is asymptotically better than the former approximation ratios
due to Peleg et al and Srivastav et al.
Furthermore we consider the problem in hypergraphs
${\cal H}=(V, {\cal E})$ with $\ell \leq (1+\epsilon)\bar{\ell}$ 
and do not assume that $\ell$ and $\Delta$ are constants.
We give a polynomial-time approximation algorithm
with an approximation ratio of $\frac{5}{6}\left(1-\frac{1}{2\ell}\right)\delta$ for any fixed $\epsilon\in [0,\frac{1}{2}]$.
The main progress is that our approximation ratio of at most 
$\frac{5}{6}\delta$.
Hence we disprove the conjecture of peleg et al for a large and important class of hypergraphs. Note that uniform hypergraphs fulfill the condition $\ell \leq (1+\epsilon)\bar{\ell}$.

\renewcommand{\baselinestretch}{1.75}
\begin{table}[h]
\begin{center}
Fundamental results and approximations for \sbmultcov problem \\\vskip 0.2cm
  \begin{tabular}{|c |c | c |  }
\hline
Hypergraph  & Approximation ratio  \\
\hline
- &
 $ H(\ell)$\cite{Vazirani01}
 \\\hline
bounded  $\ell$  &
 $H(\ell) -\frac{1}{6}$ \cite{Fujito}
 \\\hline
- &
  $\delta$ \cite{HH86,PSW93}
   \\\hline
- &
$(1-(\frac{c}{n})^\frac{1}{\delta})\cdot \delta$   where $ c>0$ is a constant. \cite{PSW97}
 \\ \hline
$l \in \mathcal{O}\left(\max\{(nb)^\frac{1}{5},n^\frac{1}{4}\}\right)$
 & $\left(1 - \frac{11(\Delta - b)}{72\ell} \right)\cdot \delta $  \cite{EMS16} 
  \\\hline
 - & $ \max\left\{ \frac{148}{149}\delta, \left(1- \frac{(b-1) e^{\frac{\delta}{4}}}{94\ell} \right)\delta \right\}$ 
  $\ $ (this paper) 
\\ \hline
 $\ell \leq (1+\epsilon)\bar{\ell}$, for $\epsilon\in [0,\frac{1}{2}]$ & $\frac{5}{6}\left(1-\frac{1}{2\ell}\right)\delta$ (this paper)
 \\ 
  \hline
\end{tabular}
\end{center}
\end{table}
\renewcommand{\baselinestretch}{1.0}

\noindent {\bf Outline of the paper.} In Section 2 we give all the definitions and tools needed for the analysis of the performed results. Section 3 we present a randomized algorithm of hybrid type and its analysis. Section 4 we give a deterministic algorithm based on matching/covering duality and its analysis. Finally we sketch some  open questions.  
\section{ Definitions and preliminaries}

Let ${\cal H}=(V,{\cal E})$ be a hypergraph, $V$ and ${\cal E}$ is the set of vertices and  hyperedges respectively. For every vertex $v \in V$ we define
the vertex degree of $v$ as
$d(v):= |\{E \in {\cal E} \mathrel{|} v \in E \}|$ and $\Gamma(v):=\{E\in \mathcal{E} \mathrel{|} v\in E \}$ 
the set of edges incident 
to $v$.
The maximum vertex degree is $\Delta:=\max_{v \in V} d(v) $.
Let $l$ denote the maximum cardinality of a hyperedge from ${\cal E}$. 
It is convenient to order the vertices and edges, i.e., $V=\{v_1,\dotsc,v_n\}$
and ${\cal E}=\{ E_1,\dotsc,E_m\}$, and to identify the vertices
and edges with their indices.

\vskip 0.1cm\noindent
\sbmultcov problem:\\
\noindent Let $\mathcal{H}=(V,\,\mathcal{E})$ be a hypergraph
and $ (b_1,b_2,\dots,b_n) \in \mathbb{N}_{\geq 2}^{n}$.
We call $C \subseteq \mathcal{E}$ a set multicover if every vertex $i\in V$ 
is contained in at least $b_i$ hyperedges of $C$. \sbmultcov
is the problem of finding a set multicover with minimum cardinality.

\vskip 0.1cm\noindent
For the later analysis we will use the following Chernoff–Hoeffding Bound inequality for a sum of independent random variables: 
\begin{theorem}[see 
\cite{McD89}]\label{AnVa}
Let $X_1,\ldots, X_n$ be independent $\{0,1\}$-random variables.
Let $X= \sum_{i=1}^n X_i$. For every $ 0<\beta\leq 1$ we have
\begin{equation*} \Pr[X\geq (1+\beta)\cdot \mathbb{E}(X)]\leq
\exp{\left(-\frac{\beta^2\mathbb{E}(X)}{3 } \right)}.
\end{equation*}
\end{theorem}
\vskip 0.2cm\noindent
A further useful concentration theorem we will use is the
Chebychev-Cantelli inequality:
\begin{theorem}[see \cite{MR}, page 64]\label{Che-Can}
Let $X$ be a non-negative random variable with finite mean $\mathbb{E}(X)$
and variance {\rm Var}$(X)$. Then for any $a>0$ it holds that
\begin{eqnarray*}
 \Pr(X\leq \mathbb{E}(X)-a)&\leq & \frac {{\rm Var} (X)}{{\rm Var}(X)+a^2}\cdot
\end{eqnarray*}
\end{theorem}
\begin{definition} 
let ${\bf k}\in \mathbb{N}_{0}^{n}$.
A ${\bf k}$-matching in ${\cal H}$ is a set $M\subseteq {\cal E}$
such that in no vertex $v_{i}$ of ${\cal H}$ more than $k_{i}$, $i\in [n]$ hyperedges from $M$ are incident. The ${\bf k}$-matching problem is to find a maximum cardinality ${\bf k}$-matching.
$\nu_{{\bf k}}({\cal H})$ denotes this maximum cardinality, and is called the ${\bf k}$-matching number of ${\cal H}$.
\end{definition}

We need the following duality theorem from combinatorics:
\begin{theorem}[Ray-Chaudhuri, 1960 \cite{B73}]\label{RC}
Consider a hypergraph $\mathcal{H}=(V,\,\mathcal{E})$ with vertex degree $d(v)$ for every $v\in V$.
let ${\bf b}, {\bf k} \in \mathbb{N}_{0}^{n}$ such that for every $v\in V$, $b_{v}+k_{v}=d(v)$.
A subset of hyperedges $M_{{\bf k}}$ is a ${\bf k}$-matching in $\mathcal{H}$ if and only if the
subset $S_{{\bf b}}:=\mathcal{E}\setminus M_{{\bf k}}$ is a ${\bf b}$-set multicover in $\mathcal{H}$. Furthermore $M_{{\bf k}}$ is of
maximum cardinality if and only if $S_{{\bf b}}$ is of minimum cardinality.
\end{theorem}
{\bf Remark 1}. Note that Theorem \ref{RC} holds also for hypergraphs with multi-sets i.e., hypergraphs with multiple hyperedges.\\

\section{ The randomized rounding algorithm }
Let ${\cal H}=(V,{\cal E})$ be a hypergraph with maximum vertex degree $\Delta$ and maximum edge size $\ell$.
An integer, linear programming formulation of \sbmultcov problem is the following:
\begin{eqnarray*}
 &&\min \sum_{j=1}^m x_j,\\
\mbox{ILP}(\Delta, {\bf b}):\qquad  &&\sum_{j=1}^m a_{ij}x_j\geq b_{i} \quad
 \mbox{ for all } i \in [n],\\
  &&x_j\in \{0,1\}\quad\mbox{ for all } j\in [m],
  \end{eqnarray*}
where $ A=(a_{ij})_{i \in [n], \,j \in [m]}\in \{0,1\}^{n \times m}$
is the vertex-edge incidence matrix of ${\cal H}$
and ${\bf b}=(b_1,b_2,\dots,b_n) \in \mathbb{N}_{\geq 2}^{n}$
is the given integer vector.
We define $b:=\min_{i\in [n]}b_{i}$ and $\delta=\Delta-b+1$.\\
The linear programming relaxation LP($\Delta,\,{\bf b}$)
of ILP($\Delta,\,{\bf b}$) is given by relaxing the integrality constraints
to  $x_j\in [0,1] \;$for all $j \in [m]$.
Let $\mathrm{Opt}$ resp. ${\rm Opt}^{*}$
be the value of an optimal solution to ILP($\Delta,\,{\bf b}$)
resp. LP($\Delta,\,{\bf b}$).
Let $(x^{\ast}_{1},\cdots,x^{\ast}_{m})$ be the optimal solution of the LP($\Delta,\,{\bf b}$).
So ${\rm Opt}^{*} =\sum_{j=1}^mx^{*}_{j}$
and ${\rm Opt}^{*} \leq \mathrm{Opt}$.

The next lemma shows that the $b_i$ greatest  values of the LP variables correspondent to the incident edges  for any vertex $v_i$ are all  greater than or equal to $\frac{1}{\delta}$.
\begin{lemma} [see \cite{PSW93}] \label{lemma:1}
Let $b_i,d,\Delta, n \in\mathbb{N} $ with $2\leqslant b_i\leqslant d-1\leqslant \Delta -1, i\in [n]$ .  Let $x_j\in [0,1],j\in [d]$,  such that   $\displaystyle  \sum_{j=1}^{d}x_j\geqslant b_i$.
Then at least $b_i$ of the $x_j$ fulfill the inequality $x_j\geqslant \frac{1}{\delta}$.
\end{lemma}

Our second lemma shows that the $b_i-1$ greatest  values of the LP variables correspondent to the incident edges  for any vertex $v_i$ are all  greater than or equal to $\frac{2}{\delta+1}$ and with lemma \ref{lemma:1} we summarize about the $b_i$ greatest  values of the LP variables correspondent to the incident edges  for any vertex $v_i$.
\begin{lemma}\label{lemma:2}
Let $b_i,d,\Delta, n \in\mathbb{N} $ with $2\leqslant b_i\leqslant d-1\leqslant \Delta -1, i\in [n]$ .  Let $x_j\in [0,1],j\in [d]$,  such that   $\displaystyle  \sum_{j=1}^{d}x_j\geqslant b_i$.
Then at least $b_{i}-1$ of the $x_j$ fulfill the inequality $x_j\geqslant \frac{2}{\delta+1}$ and exists an element $x_j$ distinct of them all who fulfill the inequality $x_j\geqslant \frac{1}{\delta}$ .
\end{lemma}

\begin{proof}
W.L.O.G  we suppose $x_1\geq x_2\geq \dots \geq x_{b_i }\geq \dots\geq x_d$.\\
Hence $ b_i-2\geq   \displaystyle  \sum_{j=1}^{b_i-2}x_j$
and  $ (d -b_i+2) x_{b_{i-1}}\geq   \displaystyle  \sum_{j=b_i-1}^{d}x_j$\\
Then
\begin{eqnarray*}
b_i-2+ (\Delta -b+2) x_{b_i-1} &\geq &b_i-2+ (\Delta -b_i+2) x_{b_i-1}\\
&\geq &b_i-2+ (d -b_i+2) x_{b_i-1}\\
  &\geq&  \sum_{j=1}^{b_i-2}x_j+  \sum_{j=b_i-1}^{d}x_j 
 =  \displaystyle  \sum_{j=1}^{d}x_j \\
 &\geq & b_i
\end{eqnarray*}
So we have
$x_{b_i-1}\geq \frac{2}{\delta+1}$\\
Since for all $j\in [b_i-1]\; ; \ x_{j} \geq x_{b_i-1 }$
then for all $j\in [b_i-1]\;  ;\;  x_{j}\geq \frac{2}{\delta+1}$.\\
Furthermore by lemma $1$  and the assumption on the orders of the  variables $x_j$, for all $j\in [b_i]\;$  we have  $x_{j}\geq \frac{1}{\delta}$
and particularly 
$x_{b_i}\geq \frac{1}{\delta}$.
\end{proof}
\begin{corollary}\label{corollarystricte}
Let $\mathcal{H}$ be a hypergraph and $(x^{\ast}_{1},\cdots,x^{\ast}_{m})$ be the optimal solution of the LP($\Delta,\,{\bf b}$). Then
$C:=\{E_j \in \mathcal{E} \mathrel{|} x_j^{\ast}\geq\frac{1}{\delta}\}$ is a set multicover  such that $|C|<\delta \mathrm{Opt}$.
\end{corollary}

\begin{proof}
 Clearly with lemma  \ref {lemma:1}, $C$ is a feasible set multicover.
 \\
Let $C_1:=\{E_j \in \mathcal{E} \mathrel{|} x_j^{\ast}\geq\frac{2}{\delta +1}\}$
and $C_2:=\{E_j\in \mathcal{E}\mathrel{|}\frac{1}{\delta} \leq x_j^{\ast}<\frac{2}{\delta +1}\}$.\\
Note that 
$C=C_1\cup C_2$ and $ C_1\cap C_2=\emptyset $, so we have $|C_1|+|C_2|=|C|$.
\\
Let $ E_j\in C_1\ $ Then $  x^*_j  \geq \frac{2}{\delta+1}$.
Since $\delta >1$ we have $ \frac{2\delta}{\delta+1}> 1$
 \\ Hence
\begin{equation*}
\delta {\rm Opt}^{*}= \sum_{j=1}^m\delta x^*_j
\geq  \displaystyle  \sum_{E_j\in C_1}\delta x^*_j+ \sum_{E_j\in C_2}\delta x^*_j
\end{equation*}
 From this we can immediately deduce
 \begin{equation}\label{inequality}
 \delta {\rm Opt}^{*}\geq  \frac{2\delta}{\delta+1}| C_1|+| C_2|
 \end{equation}
 Hence
 \begin{equation*}
 \delta \mathrm{Opt}^{*} > | C_1|+| C_2|= |C|
  \end{equation*}
Then
$$|C| < \delta \mathrm{Opt}.$$
\end{proof}

\subsection{The algorithm } 
In this section we present an algorithm with conditioned randomized rounding based on the properties satisfied by the two sets $C_1$ and $C_2$.

\begin{algorithm}
\label{alg:msetcover1}
\caption{ SET $b$-MULTICOVER}
\SetKwInOut{Input}{Input}\SetKwInOut{Output}{Output}
\Input{ a hypergraph $\mathcal{H}=(V,\, \mathcal{E})$ with maximum degree $\Delta$ and maximum hyperedge size $\ell$.\\
Let $b_i\in \mathbb{N}_{\geq 2}\text{ for }  i\in[n]\;$ ;  $\;  b:=\min_{i\in [n]}b_{i}$ and $\delta=\Delta-b+1$.}
\Output{A  set multicover $C$}
   \begin{enumerate}
     \item Initialize $C:=C_1=C_2=\emptyset$. Set $\lambda=\frac{\delta+1}{2}\;$ ;  $\alpha=\frac{(b-1)\delta e^{\frac{\delta}{4}}}{47\ell}$ and  $t=73$.
     \item  Obtain an optimal solution $x^* \in [0,1]^m$ by solving the LP($\Delta,\,{\bf b}$)  relaxation.
     \item \text{Set} $C_1:=\{E_j \in \mathcal{E} \mathrel{|} x_j^{\ast}\geq\frac{1}{\lambda}\}$ ,$\ C_2:=\{E_j \in \mathcal{E} \mathrel{|} \frac{1}{\lambda}> x_j^{\ast}\geq\frac{1}{\delta}\}$ \\
\text{and} $C_3:=\{E_j\in \mathcal{E}\mathrel{|} 0<x_j^{\ast}<\frac{1}{\lambda}\}$.
\item Take all edges of the set $C_1$ in the cover $C$.
\item if $|C_1|\geq \alpha \cdot \mathrm{Opt}^{*}$ or $t|C_1|\geq |C_2| $ then return $C=C_1\cup C_2$ . Else\\
 \begin{enumerate}
        \item \text{(Randomized Rounding) For all edges} $E_j\in C_3$
\text{include the edge} $E_j$ \text{in}\\
      \text{the cover} $C$,
\text{independently for all such} $E_j$, \text{with probability} $\lambda x_j^*$.
     \item \text{(Repairing) Repair the cover} $C$\text{(if necessary) as follows:}
     \text{Include arbitrary}\\ \text{edges from} $C_3$,
     \text{incident to vertices} $i\in [n]$ \text{ not covered by $b_i$ edges,}
     \text{to} $C$
     \text{until all}
     \text{vertices  are fully covered.}
     \item \text{Return the cover} $C$.
   \end{enumerate}
    \end{enumerate}
\end{algorithm}

Let us give a brief explanation of the ingredients of the algorithm SET $b$-MULTICOVER. \\
We start with an empty set C, which will be extended to a feasible set multicover. First we solve the LP-relaxation LP($\Delta,\,{\bf b}$) in polynomial time. Let $\alpha=\frac{(b-1)\delta e^{\frac{\delta}{4}}}{47\ell}$ and  $t=73$. The rest of the
action depends on the following two cases.\\
$\bullet$ If $|C_1|\geq \alpha \cdot \mathrm{Opt}^{*}$ or $t|C_1|\geq |C_2| $:
we pick in the cover $C$ all edges of the two sets $|C_1|$ and $|C_2|$.\\
Recall that by lemma \ref{lemma:1} the $C=C_1 \cup C_2$ is a feasible set multicover.

$\bullet$ If $|C_1|< \alpha \cdot \mathrm{Opt}^{*}$ and $t|C_1|<|C_2| $: we use LP-rounding with randomization on the edges of the set $C_3$, every edge of $C_3$ is independently picked in the cover with probability $\frac{\delta+1}{2} x_j^*$. To guarantee a feasible cover we proceed for a step of repairing.
\subsection{ Analysis of the algorithm} 

\textbf{Case $ \mathbf{|C_1|\geq \alpha \cdot \mathrm{Opt}^{*}}$ or $\mathbf{ t|C_1|\geq |C_2|} $ .}
\begin{theorem} \label{theoremcase1}
Let $\mathcal{H}$ be a hypergraph with maximum vertex degree $\Delta$ and maximum edge size $\ell$. Let $\alpha=\frac{(b-1)\delta e^{\frac{\delta}{4}}}{47\ell}$ and  $t=73$ as defined in algorithm \ref{alg:msetcover1}, if $|C_1|\geq \alpha \cdot \mathrm{Opt}^{*}$ or $t|C_1|\geq |C_2| $ , the algorithm \ref{alg:msetcover1} achieve a factor of $\left(1- \frac{(b-1) e^{\frac{\delta}{4}}}{94\ell} \right)\delta $ or $\frac{148}{149}\delta$ respectively.
\end{theorem}
\vskip 0.2cm\noindent
\begin{proof}
\textbf{Case} $\mathbf{|C_1|\geq \alpha \cdot \mathrm{Opt}^{*}}$. 
With the definition of the sets $C_1$ and  $C_2$ we have
\begin{eqnarray*}
 \delta \mathrm{Opt}^{*}= \sum_{j=1}^m\delta x^*_j
 &\geq & \displaystyle  \sum_{E_j\in C_1}\delta x^*_j+ \sum_{E_j\in C_2}\delta x^*_j \\
  &\geq &\frac{2\delta}{\delta+1}| C_1|+|C_2| \\
  &\geq &\frac{2\delta}{\delta+1}| C_1|+\left(|C|-|C_1| \right)\\
 &\geq &\frac{\delta-1}{\delta+1} |C_1|  + |C|\\
   &\overset{\delta \geq 3}{\geq} &\frac{1}{2} |C_1|  + |C|\\
   &\geq &\frac{1}{2}\alpha \cdot \mathrm{Opt}^{*}  + |C|
\end{eqnarray*}
Hence
$$|C|\leq  \left(1- \frac{ (b-1)e^{\frac{\delta}{4}}}{94\ell} \right)\delta \cdot \mathrm{Opt}^{*}$$

\textbf{Case} $ \mathbf{t|C_1|\geq|C_2|}$. We have
\begin{eqnarray*}
 t|C_1|& \geq &|C|-|C_1|
\end{eqnarray*}
Therefore
\begin{eqnarray*}
  |C_1|& \geq &\frac{1}{t+1}|C|
\end{eqnarray*}
Next with \ref{inequality} we have
\begin{eqnarray*}
 \delta \mathrm{Opt}^{*} &\geq &\frac{2\delta}{\delta+1}|C_1|+\left(|C|-|C_1| \right)\\
 &\geq &\frac{\delta-1}{\delta+1}|C_1|+|C|\\
  &\geq &\frac{\delta-1}{\delta+1}\times \frac{1}{t+1}|C|+|C|\\
  &\overset{\delta \geq 3}{\geq} & \frac{1}{2t+2}|C|+|C|
\end{eqnarray*}
Then
\begin{eqnarray*}
|C|&\leq & \frac{1}{1+ \frac{1}{2t+2}}\cdot \delta \mathrm{Opt}^{*}
 \\ &\leq &\tfrac{148}{149}\cdot \delta \mathrm{Opt}^{*}
\end{eqnarray*}
\end{proof} 
 
 \textbf{Case $ \mathbf{|C_1|< \alpha \cdot \mathrm{Opt}^{*}}$ and $\mathbf{ t|C_1|< |C_2|} $ .}
 
 Let $X_{1},...,X_{m}$ be $\{0,1\}$-random variables defined as follows:
\begin{eqnarray*}
 X_j=
 \begin{cases}
  1 &\text{if the edge}\, E_j \, \text{was picked into the cover
 before repairing}\\
  0 &\text{otherwise}.
 \end{cases}
\end{eqnarray*}
Note that the $X_1,...,X_m$ are independent for a given $x^* \in [0,1]^m$. 
For all  $i\in[n] $ we define the $\{0,1\}$- random variables $Y_{i}$ as follows:
\begin{eqnarray*}
 Y_{i}= 
 \begin{cases}
  1 &\text{if the vertex}~ v_{i}~\text{is fully covered before repairing}\\
  0 &\text{otherwise}.
 \end{cases}
\end{eqnarray*}
 We denote  $X:=\sum_{j=1}^m X_j$ and $Y:=\sum_{i=1}^n Y_i$ respectively the cardinality of the cover and the cardinality of vertices fully covered before the step of repairing. At this step by lemma \ref {lemma:2}, one more edge for each vertex is at most needed to be fully covered. The cover denoted by C obtained by the algorithm \ref{alg:msetcover1} is bounded by
\begin{equation}\label{expection}
\abs C \leq X+n-Y.
\end{equation}
\noindent Our goal by the next lemma is to estimate the expectation of the random variable $X$ so is the expectation and variance of the random variable $Y$ for the proof of the theorem \ref{Th5}. This is a restriction of Lemma $4$ in \cite{EMS16} to the last case in algorithm \ref{alg:msetcover1}.
\begin{lemma}\label{liteneq}
Let $l$ and $\Delta$ be the maximum size of an edge
resp. the maximum vertex degree, not necessarily constants.
Let $\alpha >0$, $t >0$ and $\lambda=\frac{\delta+1}{2}$
as in \mbox{Algorithm \ref{alg:msetcover1}}. In case $|C_1|< \alpha \cdot \mathrm{Opt}^{*}$ and $t|C_1|< |C_2| $ we have
\begin{description}
\item $\mathrm{(i)}$ $\mathbb{E}(Y)\geq (1-e^{-\lambda})n$.
\smallskip
\item $\mathrm{(ii)}$ ${\rm Var}(Y)\leq n^2\left(1-(1-e^{-\lambda})^2\right)$.
\smallskip
\item $\mathrm{(iii)}$
$1+\frac{t}{2}< \mathbb{E}(X)\leq \lambda \mathrm{Opt}^{*}$.
\smallskip
\item $\mathrm{(iv)}$
$ \frac{(b-1)n}{\alpha \ell} <  \mathrm{Opt}^{*}$.
\end{description}
\end{lemma}
\begin{proof}
(i) Let $i \in [n]$, $r = d(i) - b_i + 1$.
If $\abs{C_1 \cap \Gamma(v_i)} \geq b_i$, then the vertex $v_i$ is fully covered and $\Pr(Y_{i}=0)=0$.
Otherwise we get by Lemma \ref {lemma:2}
$\abs{C_1 \cap \Gamma(v_i)} = b_i-1$ and $\sum_{E_j \in (\Gamma(v_i) \cap C_3)} x_j^* \geq 1$. Therefore
\begin{eqnarray*}
\Pr(Y_i = 0) &=& \prod_{E_j \in (\Gamma(v_i) \cap C_3)}(1- \lambda x_j^*) \\
&\leq&\prod_{E_j \in (\Gamma(v_i) \cap C_3)} e^{-\lambda x_j^* }
=  e^{-\lambda \sum_{E_j \in (\Gamma(v_i) \cap C_3)} x_j^* }\\
&\leq& e^{-\lambda} .
\end{eqnarray*}
Then
\begin{align*}
\mathbb{E}(Y) &= \sum_{i=1}^n \Pr(Y_i = 1)
= \sum_{i=1}^n (1 - \Pr(Y_i = 0))\\
&\geq \sum_{i=1}^n (1 - e^{-\lambda})\\
&\geq (1-e^{-\lambda})n.
\end{align*}

(ii)
 We have
\begin{eqnarray*}
Y=\sum_{i=1}^{n}Y_{i} & \Rightarrow & Y \leq n\\
& \Rightarrow & Y^2\leq n^2 \\
& \Rightarrow & \mathbb{E}(Y^2)\leq n^2.
\end{eqnarray*}
Then
\begin{align*}
{\rm Var}(Y) &= \mathbb{E}(Y^2)-\mathbb{E}(Y)^2
\leq  n^2 -(1-e^{-\lambda})^2n^2 \\
&\leq n^2\left(1-(1-e^{-\lambda})^2\right).
\end{align*}

(iii) 
By using the LP relaxation and the definition of the sets
$C_1$ and $C_3$, and since $\lambda x^*_j\geq 1$ for all ${E_j\in C_1}$, we get
\begin{eqnarray*}
\mathbb{E}(X)&=& |C_1|+
\sum_{E_j\in C_3}\lambda x^*_j\\
&\leq &\lambda\sum_{E_j\in C_1} x^*_j + \lambda\sum_{E_j\in C_3} x^*_j\\
&\leq & \lambda\sum_{E_j\in \mathcal{E}} x^*_j \\
&\leq &  \lambda \mathrm{Opt}^{*} .
\end{eqnarray*}

  Now we can get a lower bound for the expectation of $X$
\begin{eqnarray*}
\mathbb{E}(X)= |C_1|+\sum_{E_j\in C_3}\lambda x^*_j
& \overset{ C_2\subset C_3}{\geq} & |C_1|+\lambda \sum_{E_j\in C_2} x^*_j\\
& \geq & |C_1|+\lambda \sum_{E_j\in C_2} \frac{1}{\delta}\\
& \geq & |C_1|+ \frac{\lambda}{\delta}|C_2| \\
& > & |C_1|+ \frac{1}{2}|C_2| \\
& \overset{t|C_1| < |C_2|}{>} & |C_1|+ \frac{t}{2}|C_1| \\
& > & \left(1+ \frac{t}{2} \right) |C_1|\\
& \overset{ |C_1|\geq 1}{>} & 1+ \frac{t}{2}
\end{eqnarray*}
Therefore
\begin{eqnarray*}
1+ \frac{t}{2} < \mathbb{E}(X).
\end{eqnarray*}
(iv)
 Let us consider $\tilde{\mathcal{H}} $
   the subhypergraph induced by $C_1$ in witch degree equality gives\\
  \begin{equation*}\label{double-countingSsuperieur}
  \sum_{i\in V}d(i)=\sum_{E_j\in C_1}|E_j|.
\end{equation*}\\
Since the minimum vertex degree in the subhypergraph 
$\tilde{\mathcal{H}} $ is $b-1$ with $b:=\min_{i\in [n]}b_{i}$,
 we have
$$ (b-1)n\leq \sum_{i\in V}d(i)=\sum_{E\in C_1}|E_j|\leq \ell |C_1|$$
Therefore
\begin{equation*}
 \frac{(b-1)n}{\ell}\leq | C_1|.
\end{equation*}
 With  $|C_1|< \alpha \cdot \mathrm{Opt}^{*}$ we obtain
\begin{eqnarray*}
\frac{(b-1)n}{\alpha \ell} < \mathrm{Opt}^{*}.
\end{eqnarray*}
\end{proof}
\begin{theorem}\label{Th5}
Let $\mathcal{H}$ be a hypergraph with maximum vertex degree $\Delta$ and maximum edge size $\ell$. Let $\alpha=\frac{(b-1)\delta e^{\frac{\delta}{4}}}{47\ell}$  and $t=73$ as in algorithm \ref{alg:msetcover1}. In case $|C_1|< \alpha \cdot \mathrm{Opt}^{*}$ and $t|C_1|< |C_2| $, the algorithm \ref{alg:msetcover1} returns a set multicover $C$ such that 
$$|C| < \frac{15\delta+14}{20} \cdot \mathrm{Opt}$$
with probability greater than $0.53$.
\end{theorem}
\vskip 0.2cm\noindent
\begin{proof}
Suppose $|C_1|< \alpha \cdot \mathrm{Opt}^{*}$
such that $\alpha=\frac{(b-1)\delta e^{\frac{\delta}{4}}}{47\ell}$.With $\gamma=2ne^{-\frac{\lambda}{2}}$ we have
\begin{eqnarray*}
\Pr\left(Y \leq n(1- e^{-\lambda})-\gamma \right)
& \leq & \Pr\left(Y \leq \mathbb{E}(Y)-\gamma \right)\\
& \overset{\textrm{Th } \ref{Che-Can}}{\leq} & \frac{{\rm Var}(Y)}{ {\rm Var}(Y)+\gamma^2}\\
& \leq & \frac{1}{ 1+\frac{\gamma^2}{{\rm Var}(Y)}}\\
& \leq & \frac{1}{ 1+\frac{\gamma^2}{n^2\left(1-(1-e^{-\lambda})^2\right)}}\\
& \leq & \frac{1}{ 1+\frac{\gamma^2}{n^2\left( 2e^{-\lambda}-e^{-2\lambda} \right)}}\\
& \leq & \frac{1}{1+\frac{ \gamma^2}{2n^2e^{-\lambda}}}\\
& \leq & \frac{1}{1+2}\\
 & \leq & \frac{1}{3}
\end{eqnarray*}
Choosing $\beta=\frac{2}{5}$ and $t=73$ we obtain
\begin{eqnarray*}
\Pr\left(X \geq \frac{7}{10}\left(\delta+1\right)\mathrm{Opt}^{*} \right)&=&
\Pr\left(X \geq (1 + \beta)\cdot\frac{\delta+1}{2} \mathrm{Opt}^{*} \right)\\
&=&\Pr\left(X \geq (1 + \beta)\cdot\lambda \mathrm{Opt}^{*} \right)\\
&\overset{\textrm{Lem } \ref{liteneq} (iii)}{\leq}&
\Pr\left(X \geq (1+\beta) \mathbb{E}(X)\right)\\
&\overset{\textrm{Th } \ref{AnVa}}{\leq}&
\exp\left(- \frac{\beta^2 \mathbb{E}(X)}{3}\right)\\
&\overset{\textrm{Lem } \ref{liteneq} (iii)}{\leq}&
\exp\left(- \frac{\beta^2 (1+\frac{1}{2}t)}{3}\right)\\
&\leq &
\exp\left(-2\right).
\end{eqnarray*}
Therefore it holds that

$
\Pr\left(X \leq \frac{7}{10}\left(\delta+1\right) \mathrm{Opt}^{*} \text{\ and\ }
Y \geq n(1- e^{-\lambda})-\gamma \right)\geq  1-\left(\frac{1}{3}+\exp\left(-2\right)\right).
$

Since we have
\begin{eqnarray*}
 ne^{-\lambda} +\gamma&=& ne^{-\lambda}+ 2ne^{-\frac{\lambda}{2}}  \\
 &\leq & 3ne^{-\frac{\lambda}{2}} \\
 &= & 3ne^{-\frac{\delta+1}{4}} \\
&\leq & \frac{\delta}{20}\cdot \frac{n(b-1)}{\ell}\cdot \frac{47\ell}{(b-1)\delta e^{\frac{\delta}{4}} }\\
&\leq & \frac{\delta}{20}\cdot \frac{n(b-1)}{\alpha \ell}\\
&\overset{\textrm{Lem } \ref{liteneq} (iv)}{\leq} &  \frac{\delta}{20}\cdot \mathrm{Opt}^{*} 
\end{eqnarray*}
Then
\begin{eqnarray*}
\Pr\left(|C|\leq \left(\frac{15\delta+14}{20}\right)\cdot \mathrm{Opt}^{*}\right)
& = &
\Pr\left(|C|\leq \left(\frac{7}{10}\left(\delta+1\right)+\frac{\delta}{20}\right)\cdot \mathrm{Opt}^{*}\right)\\
&\overset{(\ref{expection})}{\geq} & \Pr\left(X+n-Y\leq\left( \frac{7}{10}\left(\delta+1\right)+\frac{\delta}{20}\right)\cdot \mathrm{Opt}^{*}\right)\\
&\geq & \Pr\left(X \leq \left( \frac{7}{10}\left(\delta+1\right)\right)\cdot \mathrm{Opt}^{*} \text{\ and\ }
n-Y \leq \frac{\delta}{20}\cdot \mathrm{Opt}^{*} \right)\\
&\geq & \Pr\left(X \leq \left( \frac{7}{10}\left(\delta+1\right)\right)\cdot \mathrm{Opt}^{*} \text{\ and\ }
Y \geq n-\frac{\delta}{20}\cdot \mathrm{Opt}^{*} \right)\\
&\geq & \Pr\left(X \leq \left( \frac{7}{10}\left(\delta+1\right)\right)\cdot \mathrm{Opt}^{*} \text{\ and\ }
Y \geq n-ne^{-\lambda}-\gamma \right)\\
&\geq & \Pr\left(X  \leq \left( \frac{7}{10}\left(\delta+1\right)\right)\cdot\mathrm{Opt}^{*} \text{\ and\ }
Y \geq n(1-e^{-\lambda})-\gamma  \right)\\
&\geq & 1-\left(\frac{1}{3}+\exp\left(-2\right)\right)\\
&\geq & 0.53
\end{eqnarray*}
\end{proof}
{\bf Remark 2.} The analysis of the algorithm \ref{alg:msetcover1} 
gave three upper bounds for our cover $C$ in cardinality .\\
Since $\frac{15\delta+14}{20}$ is smaller than $\left(1- \frac{(b-1) e^{\frac{\delta}{4}}}{94\ell} \right)\delta $ and $\frac{148}{149}\delta$ for $\delta\geq 3$, we get an approximation ratio of 
 $ \max\left\{ \frac{148}{149}\delta, \left(1- \frac{(b-1) e^{\frac{\delta}{4}}}{94\ell} \right)\delta \right\}$.
 
 As mentioned above our performed guaranty improves over the ratio presented by Srivastav et al \cite{EMS16}, and this without restriction on the parameter $\ell$.
 
 Namely, for $\delta \geq 24$ we have
 \begin{eqnarray*}
 e^{\frac{\delta}{4}}>\frac{11\times94}{72}(\delta-1)
 & \Rightarrow &
 \frac{11(\delta-1)}{72\ell}< \frac{e^{\frac{\delta}{4}}}{94\ell}\\
  &\overset{b-1\geq 1}{ \Rightarrow }&
 \frac{11(\Delta-b)}{72\ell}< \frac{(b-1)e^{\frac{\delta}{4}}}{94\ell}\\
  & \Rightarrow &
 \left(1-\frac{(b-1)e^{\frac{\delta}{4}}}{94\ell}\right) \delta
<
\left(1-\frac{11(\Delta-b)}{72\ell}\right) \delta. 
 \end{eqnarray*}

\section{The $\frac{5}{6}\left(1-\frac{1}{2\ell}\right)\delta$-Approximation for the \sbmultcov problem}
The designed algorithm in this section and its analysis relies on the lemma below.\\
We denote by $\bar{\Delta}:=\frac{1}{n}\cdot\sum\limits_{v\in V}\text{deg}(v)$ the average vertex degree in $\mathcal{H}$
  and by $\bar{b} :=\frac{1}{n}\cdot\sum\limits_{v\in V}b_v$ the average of the vector $\bb=(b_1,\cdots,b_n)$. 
\begin{lemma}\label{lem1}
Let $\mathcal{H}=(V,\mathcal{E})$ be a hypergraph and $\textbf{b}\in\mathbb{N}^V$ and $\textbf{k}\in\mathbb{N}^V_0$ such that $b_v+k_v=\textnormal{deg}(v)$ for any $v\in V$. Then $\nu_k(\mathcal{H})\leq \Big( \dfrac{\bar{\Delta}}{\bar{b}}\dfrac{\ell}{\bar{\ell}}-1 \Big)\rm{Opt}$.
\end{lemma}
\begin{proof}
Let $S^{*}$ be an optimal set multicover of $\mathcal{H}$. By definition we have 
\begin{equation}\label{eq1}
m=\frac{n\bar{\Delta}}{\bar{\ell}}
\end{equation}
and by double-counting for the pairs $(v,E)$ with $v\in E$ and $E\in S^{*}$ we get
\begin{equation}\label{eq2}
n\bar{b}=\sum\limits_{v\in V}b_v\leq \ell|S^*|,
\end{equation} 
so we get

$\nu_{k}(\mathcal{H})= m-\rm{Opt}
= \left(\frac{m}{|S^{*}|}-1\right)\cdot\rm{Opt}
\overset{\eqref{eq1}}=\left(\frac{n\bar{\Delta}}{\bar
\ell|S^*|}-1\right)\cdot\rm{Opt}
\overset{\eqref{eq2}} \leq  \left(\frac{\bar{\Delta}\ell|S^{*}|}{\bar{b}\bar{\ell}|S^{*}|}-1\right)\cdot\rm{Opt}.$

\end{proof}

Let $\mathcal{M}$ be an approximation algorithm for the $k$-matching problem
with approximation guarantee \mbox{$0 < r \leq 1$}.

\smallskip
We will use Theorem \ref{RC} to construct a set multicover in $\mathcal{H}$. let ${\bf k}\in \mathbb{N}_{0}^{n}$ with components $k_{i}=d(v_{i})-b_{i}$ for all $i\in [n]$.
 Theorem \ref{RC} says that if we can
 find a ${\bf k}$-matching $M$ in $\mathcal{H}$,
 then $S:=\mathcal{E}\setminus M$ is a set multicover in $\mathcal{H}$.

\begin{algorithm}[H]
\label{alg:msetcover}
\SetKwInOut{Input}{Input}\SetKwInOut{Output}{Output}
\Input{A hypergraph $\mathcal{H}=(V,\mathcal{E})$ with $|V|=n$, ${\bf b}\in \mathbb{N}^{n}$.}
\Output{A set multicover $S$.}
\begin{enumerate}
\item For every $v_{i}\in V$ set $k_{i}:=d(v_{i})-b_{i}$ and ${\bf k}:=(k_{1},\dots,k_{n})$.
\item Execute Algorithm $\mathcal{M}$ on input ${\cal H}$ and ${\bf k}$ and return a ${\bf k}$-matching $M$.
\item Set $S:=\mathcal{E}\setminus M$.
\item Return $S$.
\end{enumerate}
\caption{\sbmultcov Algorithm}\label{MCA}
\end{algorithm}
\begin{theorem} 
The algorithm \autoref{alg:msetcover} returns a set multicover with an approximation guarantee of 
  $(1-r) \cdot \dfrac{\bar{\Delta}\cdot\ell}{\bar{b}\cdot\bar{\ell}} + r \leq \dfrac{\bar{\Delta}\cdot\ell}{\bar{b}\cdot\bar{\ell}}$. Where $r$ is the approximation ratio given by any algorithm to find a $\kk$-matching.
\end{theorem}
Note that $\Delta \geq b$ is implicitly assumed
  since otherwise the problem is unfeasible;
  indeed, we assume $\deg_\mathcal{H}(v) \geq b_v$ for all $v$.

\begin{proof}
  
  For the solution $S$ of Algorithm 2 the following holds:
  \begin{align*}
    |S| & = |\mathcal{E}| - |M| \hspace*{4cm}  \text{operation of algorithm}  \\
    & = |\mathcal{E}| - \rm{Opt} + \rm{Opt} - |M| \\
    & = \rm{Opt} + \nu_\kk(\mathcal{H}) - |M| \hspace*{2.5cm} \text{Ray-Chaudhuri} \\
    & \leq \rm{Opt} + \nu_\kk(\mathcal{H}) - r \cdot \nu_\kk(\mathcal{H}) 
    \hspace*{1.7cm} \text{approximation guarantee of $\mathcal{M}$} \\
    & \leq \rm{Opt} + (1 - r) \cdot \nu_\kk(\mathcal{H}) \\
    & \leq \rm{Opt} + (1 - r) \cdot \Big( \frac{\bar{\Delta}}{\bar{b}}\frac{\ell}{\bar{\ell}}-1 \Big) \cdot \rm{Opt}
    \hspace*{1.cm} \text{by Lemma \ref{lem1}}  \\
    & = \Big( \frac{\bar{\Delta}}{\bar{b}}\frac{\ell}{\bar{\ell}}-r\Big(\frac{\bar{\Delta}}{\bar{b}}\frac{\ell}{\bar{\ell}}-1\Big) \Big) \cdot \rm{Opt}
    = \Big( \frac{\bar{\Delta}}{\bar{b}}\frac{\ell}{\bar{\ell}}\left(1-r \right)+ r \Big) \cdot \rm{Opt}   
  \end{align*}
  
  \begin{corollary}
  Let $\mathcal{H}=(V,\mathcal{E})$ with $\ell\leq (1+\epsilon)\bar{\ell}$ for any fixed $\epsilon \in [0,\frac{1}{2}]$.
  The algorithm \autoref{alg:msetcover} returns a set multicover with an approximation guarantee of $\frac{5}{6}\delta$.
  \end{corollary}
{\bf Proof}. 
We have for all $b\geq 3$
\begin{equation*}
5b(b-1) \leq  (b+2)(5b-9)
\overset{\Delta\geq b+2} {\leq} \Delta(5b-9).
\end{equation*}
Hence $9\Delta\leq 5b(\Delta-b+1)$
and therewith $\frac{\Delta}{b}\leq \frac{5}{9}\delta$.\\

Now let consider the class of hypergraphs with $\ell\leq (1+\epsilon)\bar{\ell}$ for any fixed $\epsilon\in [0,\frac{1}{2}]$. Let set $r({\cal H})=\frac{1}{\ell}$ than we get  \\
 \begin{align*}
\frac{\bar{\Delta}}{\bar{b}}\frac{\ell}{\bar{\ell}}\left(1-r({\cal H}) \right)+ r({\cal H}) 
&\leq  
\frac{\Delta}{b}(1+\epsilon)\left(1-\frac{1}{l} \right)+ \frac{1}{l} 
& \\
&\leq  
\left(\frac{5}{9}(1+\epsilon)\left(1-\frac{1}{l} \right)+ \frac{1}{l\delta}\right)\cdot\delta  &  \\
&\overset{\delta\geq 3} \leq  
\left(\frac{15}{18}\left(1-\frac{1}{l} \right)+ \frac{1}{3l}\right)\cdot\delta  & \\
&\leq 
\left(\frac{5}{6}-\frac{1}{2l} \right)\cdot\delta 
& \\
&\leq 
\frac{5}{6}\delta. & 
 \end{align*}                                                                                                                                                                                                                                            \end{proof}
                                                                                                                                                                                                                                                   
The assumption $r({\cal H})=\frac{1}{\ell}$ may be proved by a simple  greedy analysis that we present in the following section. 
\subsection{An $\frac{1}{\ell}$-approximation algorithm for the ${\bf k}$-matching problem}
In this section we present a greedy algorithm that constructs a $\kk$-matching in $\mathcal{H}$ with an approximation ratio of $\frac{1}{\ell}$.

\begin{algorithm}[H]
\label{alg:match}
\SetKwInOut{Input}{Input}\SetKwInOut{Output}{Output}
\Input{A hypergraph $\mathcal{H}=(V,\mathcal{E})$ with $|V|=n$ and $|{\cal E}|=m$, ${\bf k}\in \mathbb{N}_{0}^{n}$.}
\Output{A ${\bf k}$-matching $M$ in ${\cal H}$.}
 \begin{enumerate}
 \item Initialize $M:=\emptyset$. Consider any ordering $E_{1},E_{2},\cdots, E_{m}$ of the edges in ${\cal E}$.
 \item for $i=1,2,\cdots,m$ do
 \item \qquad if $M\cup\{E_i\}$ is a ${\bf k}$-matching then
 \item \qquad \qquad set $M:=M\cup\{E_i\}$.
 \item Return $M$.
\end{enumerate}
\caption{${\bf k}$-Matching Algorithm}\label{MCA}
\end{algorithm}
\begin{theorem}\label{TH0}
Algorithm \ref{alg:match} constructs a  ${\bf k}$-matching $M$ in ${\cal H}$ with $|M|\geq \frac{1}{\ell}\nu_{{\bf k}}({\cal H})$.
\end{theorem}
{\bf Proof.} It is clear by construction that $M$ is a ${\bf k}$-matching. Set $N:=|M|$ and let $M=\{f_{1},\cdots,f_{N}\}$.
Let $|M^{*}|$ be a maximum ${\bf k}$-matching in ${\cal H}$. We compare the cardinality of $M$ with the cardinality of $M^{*}$ 
by iteratively adding all $M$-edges into $M^{*}$ and removing some $M^{*}$-edges in order to fulfill the ${\bf k}$-matching condition.
Suppose we have arrived at edge $f_{i}$, $i\in [N]$. If $f_{i}\in M^{*}$, we keep $f_{i}$ in $M^{*}$. Otherwise, for every $v\in f_{i}$, for which the ${\bf k}$-matching conditions in $M^{*}\cup \{f_{i}\}$ is violated, remove one $M^{*}$-edge
incident in $v$. 
Thus at most $\ell$ $M^{*}$-edge are removed in this step. Define $r_{i}$ as the number of removed $M^{*}$-edges in step $i$, if $f_{i}\notin M^{*}$, and $r_{i}:=1$, if $f_{i}\in M^{*}$. Trivially $r_{i}\leq \ell$ for all $i\in [N]$. 
Furthermore, $\sum_{i\in [N]} r_{i}=|M^{*}|$. Assume for a moment that $\sum_{i\in [N]} r_{i}\leq |M^{*}|-1$. Then some $M^{*}\setminus M$ edges have not been removed. Let $E_{t}$ be such an edge. Since $M\cup\{E_{t}\}$  is a ${\bf k}$-matching, 
the algorithm should have included $E_{t}$ into $M$, which is a contradiction.
Summation over the number of iterations gives
\begin{equation*}|M|=\sum_{i\in [N]}1 \geq \sum_{i\in [N]} \frac{r_{i}}{\ell}=\frac{1}{\ell}\sum_{i\in [N]} r_{i}=\frac{1}{\ell}|M^{*}|=\frac{1}{\ell} \nu_{{\bf k}}({\cal H}).\end{equation*} \hfill$\Box$\\
{\bf Remark 3}. We note that, 
\begin{itemize}
\item[\rm{i})]  our approximation ratio of $\frac{1}{\ell}$ for ${\bf k}$-matching problem improves over the ratio of $\frac{1}{\ell+1}$ presented by Krysta \cite{Krysta}. This to our knowledge the best achieved result for the problem without restrictions on either on ${\bf k}$ nor on the instance.\\
\item[\rm{ii})]  Theorem \ref{TH0} will be used for the construction of a ${\bf k}$-matching in this paper. It is also possible to use other ${\bf k}$-matching approximation algorithms. 
\end{itemize}
\section{Future Work}
We believe now that the conjecture of Peleg et all holds in general setting. Hence proving the truly of the conjecture remains a big challenge for our future works.

\end{document}